\documentclass{article}  

\usepackage{amsmath}
\usepackage{amsfonts}
\usepackage{amssymb}
\usepackage{amsthm}
\usepackage{mathtools}
\usepackage{mathrsfs}
\usepackage{mathpazo}

\usepackage{graphicx}
\usepackage{enumerate,color}
\usepackage{microtype}
\usepackage[margin=15pt,font=small,labelfont={bf,sf}]{caption}
\usepackage{sectsty}
\allsectionsfont{\sffamily}

\usepackage{tikz}
\usepackage{pgfplots}
\tikzset{every picture/.style=thick}

\usepackage{hyperref}
\usepackage[nameinlink]{cleveref}

\definecolor{lightblue}{rgb}{0.54, 0.81, 0.94}
\hypersetup{
    colorlinks = true,
    linkcolor = lightblue,
    citecolor = lightblue,
    urlcolor = lightblue,
    }

\newcommand{\abs}[1]{\ensuremath{\left\vert{#1}\right\vert}}
\newcommand{\C}{\ensuremath{\mathbb{C}}}
\newcommand{\Cext}{\ensuremath{\mathbb{C}_{\infty}}}
\newcommand{\cfunof}[1]{\ensuremath{\left\{#1\right\}}}
\newcommand{\diag}[1]{\mathrm{diag}\ensuremath{\left(#1\right)}}
\newcommand{\disc}{\ensuremath{\mathscr{A}_0}}
\newcommand{\delt}[2]{\ensuremath{\mathbf{\Delta}_{#2}\left(#1\right)}}
\newcommand{\dt}{\ensuremath{\tfrac{\text{d}}{\text{d}t}}}
\DeclarePairedDelimiter\ceil{\lceil}{\rceil}
\newcommand{\funof}[1]{\ensuremath{\left(#1\right)}}
\DeclareMathOperator{\intr}{int}
\newcommand{\Hfty}{\ensuremath{\mathscr{H}_\infty}}
\newcommand{\norm}[1]{\ensuremath{\left\Vert #1 \right\Vert}}
\newcommand{\R}{\ensuremath{\mathbb{R}}}

\newcommand{\s}{\ensuremath{\left(s\right)}}
\newcommand{\sqfunof}[1]{\ensuremath{\left[#1\right]}}
\DeclareMathOperator{\trace}{tr}
\newcommand{\tm}{\ensuremath{\left(t\right)}}

\newtheorem{theorem}{Theorem}

\newtheorem{remark}{Remark}
\newtheorem{lemma}{Lemma}

\begin{document}

\title{A Generalisation of the Secant Criterion}
\author{Richard Pates\thanks{The author is a member of the ELLIIT Strategic Research Area at Lund University. This work was supported by the ELLIIT Strategic Research Area. This project has received funding from ERC grant agreement No 834142.}}

\maketitle

\begin{abstract}
The cyclic feedback interconnection of $n$ subsystems is the basic building block of control theory. Many robust stability tools have been developed for this interconnection. Two notable examples are the small gain theorem and the Secant Criterion. Both of these conditions guarantee stability if an inequality involving the geometric mean of a set of subsystem indices is satisfied. The indices in each case are designed to capture different core properties; gain in the case of the small gain theorem, and the degree of output-strict-passivity in the Secant Criterion. In this paper we identify entire families of other suitable indices based on mappings of the unit disk. This unifies the small gain theorem and the Secant Criterion, as well as a range of other stability criteria, into a single condition.
\end{abstract}


\subsection*{Notation}

$\R$ and $\C$ denote the real and complex fields respectively, and $\Cext=\C\cup\cfunof{\infty}$ the extended complex plane. $I_n$ denotes the $n\times{}n$ identity matrix. $\R\sqfunof{s}$ denotes the set of polynomials in the indeterminate $s$ with real coefficients, $\Hfty$ the set transfer functions of stable, linear time-invariant systems, and $\disc$ the set of functions in $\Hfty$ that are continuous on the extended imaginary axis. The H-infinity norm of $G\in\Hfty$ is written as $\norm{G}_\infty$. A pair $N,M\in\Hfty$ is said to be a left coprime factor representation of $G$ if $G=M^{-1}N$ and there exist $X,Y\in\Hfty$ such that $NX+MY=I_n$. The diagonal matrix with entries $x_1,\ldots{},x_n$ on the diagonal is denoted $\diag{x_1,\ldots{},x_n}$.

\section{Introduction}

The Secant Criterion is a stability condition for the cyclic feedback interconnection in \Cref{fig:1} \cite{Thr91}. It was originally introduced with biological applications in mind \cite{Thr91,TO78}, though it has since seen many generalisations and extensions relevant for other application domains, such as large-scale systems \cite{AS06,Arc11}. Just like the small gain theorem, the Secant Criterion guarantees stability of the feedback interconnection if the geometric mean of the `indices' $\gamma_1,\ldots{},\gamma_n$ of the $n$ individual subsystems 
\[
\bar{\gamma}={}\sqrt[n]{\gamma_1\gamma_2\cdots{}\gamma_n}
\]
satisfies an inequality. However there are several notable differences.
\begin{enumerate}
\item The inequality is different. The small gain theorem requires that
\begin{equation}
\label{eq:sec1}\bar{\gamma}<1.
\end{equation}
However the Secant Criterion instead requires that
\begin{equation}\label{eq:sec2}
\bar{\gamma}<\sec\funof{\tfrac{\pi}{n}}.
\end{equation}
\item The notion of `index' is different. In the small gain theorem, $\gamma_k$ is typically the induced norm of the \emph{k}th subsystem. However in the Secant Criterion, $\gamma_k$ quantifies the degree of output-strict-passivity \cite{Son06}.
\end{enumerate}
Both criteria have advantages and disadvantages with respect to the features of the subsystem dynamics that they capture. For example, the Secant Criterion exploits phase information in a way that allows \eqref{eq:sec2} to be satisfied by subsystems with larger H-infinity norms than would be permitted by the small gain theorem, potentially leading to a sharper stability analysis. However if even a single subsystem contains a pure delay, \eqref{eq:sec2} will always fail, which would not be the case for \eqref{eq:sec1}. In recent years, further closely related criteria have been discovered. These include results very much like the Secant Criterion, but instead based on input-feedforward-passivity indices \cite{YP10}, and also the so called large gain theorem \cite{ZM08}. These conditions also have the same essential features of (1) and (2), though the indices and inequalities again take on different forms.

\begin{figure}
\centering
\begin{tikzpicture}[>=stealth]
\node[rectangle, draw] (g1) at (1,0) {${G}_n$};
\node[circle, draw, minimum size=0.05cm] (s1) at
(2,0) {};
\node[rectangle, draw] (g2) at (5,0) {${G}_{2}$};
\node[circle, draw, minimum size=0.05cm] (s2) at
(4,0) {};
\node[circle, draw, minimum size=0.05cm] (s3) at
(6,0) {};
\node[rectangle, draw] (gn) at (7,0) {${G}_1$};
\node[circle, draw, minimum size=0
.05cm] (sn) at
(8,0) {};
\node[rectangle, draw] (fb) at (4.375,-1.5) {$-1$};
\draw[<-] (g1) -- node[above] {$u_n$} (s1);
\draw[->] (2.5,0) -- node[above,xshift=.25cm] {$y_{n-1}$} (s1);
\draw (s2) -- (3.5,0) node[above,xshift=.1cm] { $u_3$};
\draw (3,0) node {\Large $\mathbf{\cdots}$};
\draw[<-] (s2) -- node[above] {$y_2$} (g2);
\draw[<-] (g2) -- node[above] {$u_2$} (s3);
\draw[<-] (s3) -- node[above] {$y_1$} (gn);
\draw[<-] (gn) -- node[above] {$u_1$} (sn);
\draw[<-] (sn) -- (8.65,0) -- (8.65,-1.5) -- (fb);
\draw[<-] (fb) -- (.25,-1.5) -- (.25,0) -- node[above] {$y_n$} (g1);
\draw[->] (2,.7) node[above] {$d_n$} -- (s1);
\draw[->] (4,.7) node[above] {$d_3$} -- (s2);
\draw[->] (6,.7) node[above] {$d_{2}$} -- (s3);
\draw[->] (8,.7) node[above] {$d_1$} -- (sn);
\end{tikzpicture}

\caption{\label{fig:1}The cyclic feedback interconnection of $n$ subsystems.}
\end{figure}
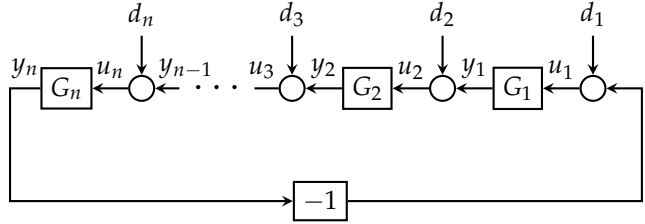

In this paper we unify and generalise these results. Our main contribution is to derive a secant like criterion that holds for general families of indices centred around mappings of the unit disk through M\"{o}bius transforms. We demonstrate that in the linear time-invariant case, all the aforementioned criteria are special cases of this result. The presented criterion also allows more general indices to be obtained, that can be tailored to different dynamical features of the subsystems. This is illustrated though a simple example in which the subsystem dynamics contain a mixture of high gain and pure delay. Finally converse results and nonlinear extensions are briefly discussed.

\section{Results}\label{sec:res}

In this section we present a generalisation of the Secant Criterion that is suitable for the analysis of the cyclic feedback interconnection of $n$ linear time-invariant single-input-single-output systems. Our central result, presented as \Cref{thm:1} in the \cref{sec:res1}, concerns the invertibility of a particular matrix defined in terms of subsets of the extended complex plane. We then show how to obtain stability criteria from this result in both the behavioral and input-output frameworks in \cref{sec:res2}. Frequency domain and state-space methods for testing the stability criteria are then presented and illustrated in \cref{sec:res3,sec:res4}. Finally, connections to existing results, and methods to define more general stability indices, are discussed in \cref{sec:res5}. Remarks connecting the results to existing literature, as well as directions for potential nonlinear extensions, are given throughout.

\subsection{A matrix invertibility result} \label{sec:res1}

Consider the following structured matrix
\begin{equation}\label{eq:maininv}
R=\begin{bmatrix}
\diag{x_1,\ldots{},x_n}&\diag{y_1,\ldots{},y_n}\\
I_n&S
\end{bmatrix},
\end{equation}
where
\begin{equation}\label{eq:cyclic}
S=\begin{bmatrix}
0&-1\\
I_{n-1}&0
\end{bmatrix}.
\end{equation}
This subsection is concerned with the invertibility of $R$ for sets of $x_k,y_k\in\C$. More specifically, we consider sets that are characterised by subsets of the extended complex plane that are obtained by first mapping the unit disk through a M\"{o}bius transform 
\[
f\funof{z}=\frac{az+b}{cz+d},
\]
and then scaling the resulting circle or half-plane by $\gamma$. The following theorem shows that given any set of scaling constants $\gamma_1,\ldots{},\gamma_n$ (one for each $\funof{x_k,y_k}$ pair), $R$ is invertible if and only if a secant like criterion is satisfied.
\begin{theorem}\label{thm:1}
Let
\[
\delt{\gamma}{\mathrm{mob}}=\cfunof{\frac{\gamma\funof{az+b}}{cz+d}\in\Cext:z\in\C,\abs{z}\leq{}1},
\]
where $a$, $b$, $c$ and $d$ are real numbers such that $ad-bc\neq{}0$. Given any positive $\gamma_1,\ldots{},\gamma_n$, the following conditions are equivalent.
\begin{enumerate}[(i)]
\item For all $x_k,y_k\in\C$ such that $x_k/y_k\in\delt{\gamma_k}{\mathrm{mob}}$, the matrix $R$ in \eqref{eq:maininv} is invertible.
\item  For $\bar{\gamma}=\sqrt[n]{\gamma_1\gamma_2\cdots{}\gamma_n}$, the inequality
\[
\begin{aligned}
\funof{b^2-a^2}\bar{\gamma}^2+2\beta\funof{ac-bd}\bar{\gamma}+d^2-c^2&>0
\end{aligned}
\]
holds with $\beta=\cos\funof{\pi/n}$ and $\beta=\cos\funof{\pi{}\funof{2\ceil{\tfrac{n}{2}}-1}/n}$.
\end{enumerate}
\end{theorem}

\begin{proof} 
Let $\Gamma=\diag{\gamma_1,\ldots{},\gamma_n}$ and $\mathbf{\Delta}$ denote the set of diagonal matrices with entries in the closed unit disk
\[
\mathbf{\Delta}=\cfunof{\diag{z_1,\ldots{},z_n}:z_k\in\C,\abs{z_k}\leq{}1}.
\]
Recall that the structured singular value (with respect to diagonal complex uncertainty) of a matrix is defined according to
\[
\mu_{\mathbf{\Delta}}\funof{Z}=\frac{1}{\min\cfunof{\abs{\gamma}:\det\funof{I_n-\gamma\Delta{}Z}=0,\Delta\in\mathbf{\Delta}}},
\]
unless no $\gamma$ exists such that $\det\funof{I_n-\gamma\Delta{}Z}=0$, in which case $\mu_{\mathbf{\Delta}}\funof{Z}=0$.  Consider the following statement.
\begin{enumerate}
\item[$\text{(i)}^*$] $\det\funof{-b\Gamma{}S+dI_n}\neq{}0$ and $\mu_\mathbf{\Delta}\funof{g\funof{\Gamma{}S}}<1$, where
\[
g\funof{Z}=\funof{-aZ+cI_n}\funof{-bZ+dI_n}^{-1}.
\]
\end{enumerate}
We will first show that (i) is equivalent to $\text{(i)}^*$. This connects (i) to  $\mu$-analysis, unlocking some techniques that will help to demonstrate that $\text{(i)}^*$ is equivalent to (ii) (the results we will require from $\mu$-analysis are fairly standard, and can all be found in \cite{PD93}).

$\text{(i)}\Longleftrightarrow{}\text{(i)}^*$: Let $X=\diag{x_1\ldots{}x_n}$ and $Y=\diag{y_1,\ldots{}y_n}$, and observe that
\[
R\begin{bmatrix}-S&0\\I_n&I_n\end{bmatrix}=\begin{bmatrix}Y-XS&Y\\0&I_n\end{bmatrix}.
\]
Since $S$ is invertible (in fact $SS^{\mathsf{T}}=I_n$), the above implies that $R$ is invertible if and only if
\[
\det\funof{Y-XS}=\det\funof{\begin{bmatrix}X&Y\end{bmatrix}\begin{bmatrix}-S\\I_n\end{bmatrix}}\neq{}0.
\]
It follows from the definition of $\delt{\gamma_k}{\mathrm{mob}}$ that $x_k/y_k\in\delt{\gamma_k}{\mathrm{mob}}$ if and only if there exists a $\delta_k\in\C$ and a $\lambda_k\in\C$ such that
\[
\abs{\delta_k}\leq{}1,\,\lambda_k\neq{}0,\,\text{and}\,\lambda_k\begin{bmatrix}x_k&y_k\end{bmatrix}=\begin{bmatrix}\delta_k&1\end{bmatrix}\begin{bmatrix}a&c\\b&d\end{bmatrix}\begin{bmatrix}\lambda_k&0\\0&1\end{bmatrix}.
\]
Putting $\Delta=\diag{\delta_1,\ldots{},\delta_n}$ and $\Lambda=\diag{\lambda_1,\ldots{},\lambda_n}$ shows that
\[
\begin{bmatrix}X&Y\end{bmatrix}=\Lambda^{-1}\begin{bmatrix}\Delta&I_n\end{bmatrix}\begin{bmatrix}aI_n&cI_n\\bI_n&dI_n\end{bmatrix}\begin{bmatrix}\Gamma{}&0\\0&I_n\end{bmatrix}.
\]
It follows from the above that (i) holds if and only if
\begin{equation}\label{eq:1}
\det\funof{\begin{bmatrix}\Delta&I_n\end{bmatrix}\begin{bmatrix}-a\Gamma{}S+cI_n\\-b\Gamma{}S+dI_n\end{bmatrix}}\neq{}0
\end{equation}
for all $\Delta\in\mathbf{\Delta}$. Since $0\in\mathbf{\Delta}$, we see from the above that (i) implies that $\det\funof{-b\Gamma{}S+dI_n}\neq{}0$. Multiplying \eqref{eq:1} from the right by $\det(\funof{-b\Gamma{}S+dI_n}^{-1})$ then shows that $\text{(i)}$ implies that
\[
\det\funof{I+\Delta{}g\funof{\Gamma{}S}}\neq{}0
\]
for all $\Delta\in\mathbf{\Delta}$. Since for all $\theta\in\R$, 
\[
\Delta\in\mathbf{\Delta}\Longleftrightarrow{}\exp\funof{i\theta}\Delta\in\mathbf{\Delta},
\]
the above implies that
\[
\det\funof{I-\gamma\Delta{}g\funof{\Gamma{}S}}\neq{}0
\]
for all $\abs{\gamma}\leq{}1$ and $\Delta\in\mathbf{\Delta}$, and so (i)$\,\implies\text{(i)}^*$. Reversing these final steps shows that under the hypothesis of $\text{(i)}^*$, \eqref{eq:1} holds, and so $\text{(i)}^*\!\!\implies$(i), establishing the desired equivalence.

$\text{(i)}^*\implies{}\text{(ii)}$: We will proceed by showing that $\mu_{\mathbf{\Delta}}\funof{g\funof{\Gamma{}S}}$ can be determined analytically in terms of $a,b,c,d$ and $\bar{\gamma}$. The statement in (ii) will then follow by rewriting the inequality in $\text{(i)}^*$ in terms of this description. First introduce the set of diagonal positive semi-definite matrices with unit trace
\[
\mathcal{D}=\cfunof{D:D=\diag{d_1,\ldots{},d_n},d_k\geq{}0,\trace\funof{D}=1}.
\]
We will determine $\mu_{\mathbf{\Delta}}\funof{g\funof{\Gamma{}S}}$ based on the following bounds
\[
\rho\funof{g\funof{\Gamma{}S}}\leq{}\mu_{\mathbf{\Delta}}\funof{g\funof{\Gamma{}S}}\leq{}\inf_{D\in\intr{}\mathcal{D}}\norm{D^{\frac{1}{2}}g\funof{\Gamma{}S}D^{-\frac{1}{2}}},
\]
where $\rho\funof{\cdot}$ denotes the spectral radius and $\norm{\cdot}$ the matrix two-norm. Consider now
\[
h\funof{D}=\Gamma{}SDS^{\mathsf{T}}\Gamma{}/\trace\funof{\Gamma{}SDS^{\mathsf{T}}\Gamma{}}.
\]
Since
\begin{equation}\label{eq:4}
S\begin{bmatrix}
d_1&0&\cdots{}&0\\
0&d_2&\ddots{}&\vdots{}\\
\vdots{}&\ddots{}&\ddots{}&0\\
0&\cdots{}&0&d_n
\end{bmatrix}S^\mathsf{T}=
\begin{bmatrix}
d_n&0&\cdots{}&0\\
0&d_1&\ddots{}&\vdots{}\\
\vdots{}&\ddots{}&\ddots{}&0\\
0&\cdots{}&0&d_{n-1}
\end{bmatrix},
\end{equation}
we see that $h:\mathcal{D}\rightarrow{}\mathcal{D}$. Since $\mathcal{D}$ is compact and convex, Brouwer's fixed point theorem guarantees that there exists a $\bar{D}\in\mathcal{D}$ such that $h\funof{\bar{D}}=\bar{D}$. It further follows from \eqref{eq:4} that the diagonal entries of $\bar{D}$ are all non-zero (else by \eqref{eq:4} they must all be zero, meaning that $\bar{D}\notin\mathcal{D}$), and so $\bar{D}$ is invertible. Therefore
\[
\begin{aligned}
\bar{D}^{-\frac{1}{2}}\Gamma{}S\bar{D}^{\frac{1}{2}}\funof{\bar{D}^{-\frac{1}{2}}\Gamma{}S\bar{D}^{\frac{1}{2}}}^{\mathsf{T}}&=\trace\funof{\Gamma{}S\bar{D}S^{\mathsf{T}}\Gamma}I_n,\\
\end{aligned}
\]
which implies that $\bar{D}^{-\frac{1}{2}}\Gamma{}S\bar{D}^{\frac{1}{2}}$ is a normal matrix. Therefore 
\[
\bar{D}^{\frac{1}{2}}g\funof{\Gamma{}S}\bar{D}^{-\frac{1}{2}}=g\funof{\bar{D}^{-\frac{1}{2}}\Gamma{}S\bar{D}^{\frac{1}{2}}}
\]
is also normal, meaning that
\begin{equation}\label{eq:rhomu}
\norm{\bar{D}^{\frac{1}{2}}g\funof{\Gamma{}S}\bar{D}^{-\frac{1}{2}}}=\rho\funof{g\funof{\Gamma{}S}}=\mu_{\mathbf{\Delta}}\funof{g\funof{\Gamma{}S}}.
\end{equation}
We will now obtain an analytical expression for $\rho\funof{g\funof{\Gamma{}S}}$ in terms of $a,b,c,d$ and $\bar{\gamma}$. First observe that
\[
\det\funof{\lambda{}I_n-\Gamma{}S}=\lambda^n+\gamma_1\gamma_2\cdots{}\gamma_n,
\]
which implies that the spectrum of $\Gamma{}S$ is given by
\begin{equation}\label{eq:spec}
\sigma\funof{\Gamma{}S}=\cfunof{\bar{\gamma}\exp\funof{i\pi\funof{2k-1}/n}:k=1,\ldots{},n}.
\end{equation}
Since $\sigma\funof{g\funof{\Gamma{}S}}=g\funof{\sigma\funof{\Gamma{}S}}$,
\begin{equation}\label{eq:5}
\rho\funof{g\funof{\Gamma{}S}}=\max\cfunof{\abs{g\funof{\lambda}}:\lambda\in\sigma\funof{\Gamma{}S}}.
\end{equation}
Direct substitution shows that given any $\theta\in\R$, 
\[
\abs{g\funof{\bar{\gamma}\exp\funof{i\theta}}}<1
\]
if and only if
\begin{equation}\label{eq:6}
\funof{b^2-a^2}\bar{\gamma}^2+2\funof{ac-bd}\cos\funof{\theta}\bar{\gamma}+d^2-c^2>0.
\end{equation}
Therefore by \eqref{eq:spec} and \eqref{eq:5}, \eqref{eq:6} holds for
\[
\theta\in\cfunof{\pi\funof{2k-1}/n:k=1,\ldots{},n},\]
and (ii) follows by setting $k=1$ and $k=\ceil{\tfrac{n}{2}}$.

$\text{(ii)}\implies\text{(i)}^*$: First note that for any integer $k$,
\[
\cos\funof{\pi\funof{2\ceil{\tfrac{n}{2}}-1}/n}\leq{}\cos\funof{\pi\funof{2k-1}/n}\leq{}\cos\funof{\pi/n}.
\]
Therefore given any $\theta\in\cfunof{\pi\funof{2k-1}/n:k=1,\ldots{},n}$, (ii) implies that \eqref{eq:6} holds, and so
\[
\max\cfunof{\abs{g\funof{\lambda}}:\lambda\in\sigma\funof{\Gamma{}S}}<1.
\]
This implies that $d/b\notin\sigma\funof{\Gamma{}S}$ (note that $g\funof{d/b}=\cfunof{\infty}$), which implies that $\det\funof{-b\Gamma{}S+dI_n}\neq{}0$. Therefore $g\funof{\Gamma{}S}$ exists, and so by \eqref{eq:rhomu} and \eqref{eq:5},
\[
\mu_{\mathbf{\Delta}}\funof{g\funof{\Gamma{}S}}<1,
\]
which implies $\text{(i)}^*$ as required.
\end{proof}

\subsection{Stability criteria based on \Cref{thm:1}}\label{sec:res2}

In this subsection we show show how to use \Cref{thm:1} to obtain stability criteria for the cyclic feedback interconnection illustrated in \Cref{fig:1}. We will first consider the autonomous case (i.e. no external disturbances), and derive conditions in the behavioral framework that apply whenever the subsystems have dynamics described by ordinary differential equations. We will then introduce the effect of external disturbances, and perform analysis in the input-output setting for subsystem dynamics that admit coprime factor representations. For simplicity we will restrict our attention to the linear time-invariant case throughout. Nonlinear extensions are briefly discussed in \Cref{rem:nonlinear} in \cref{sec:res3}.

\subsubsection{The autonomous case}

Suppose we have a system described by a set of differential equations on the form
\begin{equation}\label{eq:autbeh}
R\funof{\dt}w=0,
\end{equation}
where $R\in\R^{n\times{}n}\sqfunof{s}$. It is shown in \cite[Theorem 7.2.2]{PW98} that the dynamics in \eqref{eq:autbeh} are asymptotically stable, meaning that every (weak) solution to \eqref{eq:autbeh} satisfies $\lim_{t\rightarrow{}\infty}w\tm=0$, if and only if $\det\funof{R\s}\neq{}0$ for all $s$ in the closed right half-plane.

Consider now the cyclic feedback interconnection where each subsystem has dynamics described by a set of differential equations on the form
\begin{equation}\label{eq:subsys}
M_k\funof{\dt}y_k=N_k\funof{\dt}u_k,\,k=1,\ldots{},n,
\end{equation}
where $M_k,N_k\in\R\sqfunof{s}$. In the absence of disturbances (i.e. the autonomous case), the cyclic feedback corresponds to the equation
\[
u=Sy,
\]
where $S$ is as in \eqref{eq:cyclic}. Putting these equations together and writing them on the form in \eqref{eq:autbeh} gives
\begin{equation}\label{eq:cycbeh}
\begin{bmatrix}
N\funof{\dt}&M\funof{\dt}\\
I_n&S
\end{bmatrix}
\begin{bmatrix}u\\-y\end{bmatrix}=0,
\end{equation}
where
\[
\begin{aligned}
N&=\diag{N_1,\ldots{},N_n}\;\;\text{and}\\
M&=\diag{M_1,\ldots{},M_n}.
\end{aligned}
\]
Observe that the matrix in \eqref{eq:cycbeh} is on precisely the form of the structured matrix in \eqref{eq:maininv} from \cref{sec:res1}. It therefore follows immediately from \Cref{thm:1} that given any positive $\gamma_1,\ldots{},\gamma_n$, if the inequality in  \Cref{thm:1}(ii) holds, then the dynamics in \eqref{eq:cycbeh} are asymptotically stable for all $M_k,N_k\in\R\sqfunof{s}$ such that
\[
N_k\s/M_k\s\in\delt{\gamma_k}{\mathrm{mob}}
\]
for all $s$ in the closed right half-plane. As would be expected, the converse also holds, as demonstrated by the following lemma.

\begin{lemma}\label{lem:1}
Let $a,b,c,d\in\R$ satisfy $ad-bc\neq{}0$. Given any positive $\gamma_1,\ldots{},\gamma_k$ with geometric mean $\bar{\gamma}=\sqrt[n]{\gamma_1\gamma_2\cdots{}\gamma_n}$, if the inequality
\[
\begin{aligned}
\funof{b^2-a^2}\bar{\gamma}^2+2\beta\funof{ac-bd}\bar{\gamma}+d^2-c^2&\leq{}0
\end{aligned}
\]
holds for $\beta=\cos\funof{\pi/n}$ or $\beta=\cos\funof{\pi{}\funof{2\ceil{\tfrac{n}{2}}-1}/n}$, then there exist $M_k,N_k\in\R\sqfunof{s}$ such that
\begin{equation}\label{eq:mncond}
N_k\s/M_k\s\in\cfunof{\frac{\gamma\funof{az+b}}{cz+d}\in\Cext:z\in\C,\abs{z}\leq{}1}
\end{equation}
for all $s$ in the closed right half-plane, and \eqref{eq:cycbeh} is not asymptotically stable.
\end{lemma}
\begin{proof}
By \Cref{thm:1} there exist $x_k,y_k\in\C$ such that $x_k/y_k\in\delt{\gamma_k}{\mathrm{mob}}$ and
\[
\det\funof{\begin{bmatrix}
\diag{x_1,\ldots{},x_n}&\diag{y_1,\ldots{},y_n}\\
I_n&S
\end{bmatrix}
}=0.
\]
Hence the result follows if we can find $M_k,N_k\in\R\sqfunof{s}$ such that \eqref{eq:mncond} holds, and $N_k\s/M_k\s=x_k/y_k$ for some $s$ in the closed right half-plane. The following construction works with $s=i$.

From the definition of $\delt{\gamma}{\mathrm{mob}}$, there exist positive constants $\alpha_k\leq{}1$, and $\theta_k\in[-2\pi,0)$ such that
\[
x_k/y_k=\frac{\gamma_k\funof{a\alpha_k\exp\funof{i\theta_k}+b}}{c\alpha_k\exp\funof{i\theta_k}+d}.
\]
Now consider
\[
z_k\s=\alpha_k\funof{\frac{1-\phi_ks}{1+\phi_ks}}^3,
\]
where $\phi_k=\tan\funof{-\theta_k/6}$. It is easily shown that
\[
z_k\funof{i}=\alpha_k\exp\funof{i\theta_k},
\]
and also that that $\phi_k>0$ which implies that $\abs{z_k\s}\leq{}1$ for all $s$ in the closed right half-plane. Setting
\[
\begin{bmatrix}
N_k\s&M_k\s
\end{bmatrix}=\begin{bmatrix}\alpha_k\funof{1-\phi_ks}^3&\funof{1+\phi_ks}^3\end{bmatrix}
\begin{bmatrix}
\gamma_ka&c\\\gamma_kb&d
\end{bmatrix}
\]
completes the proof.
\end{proof}

\subsubsection{The input-output case}

We will now shift focus slightly, and study stability of the feedback interconnection described by the equations
\begin{equation}\label{eq:ipop}
\begin{aligned}
\hat{y}_k\s&=G_k\s{}\hat{u}_k\s,\,k=1,\ldots{},n,\\
\hat{u}\s&=S\hat{y}\s+\hat{d}\s.
\end{aligned}
\end{equation}
In the above $G_k$ is a transfer function describing the dynamics of the \emph{k}th subsystem in the cyclic feedback interconnection, $\hat{u}$, $\hat{y}$, $\hat{d}$ are the Laplace transforms of the vectors of inputs and outputs of these subsystems, and a set of external disturbances, respectively. 

The basis for our input-output analysis of the cyclic feedback interconnection is the following simple lemma concerning stability of \eqref{eq:ipop}.

\begin{lemma}\label{lem:2}
Let $G=\diag{G_1,\ldots{},G_n}$ be a transfer function with left coprime factor representation $G=M^{-1}N$, where $M,N\in\Hfty^{n\times{}n}$. Then the transfer function from $d$ to $y$ in \eqref{eq:ipop} is stable if and only if
\begin{equation}\label{eq:stabcond}
\begin{bmatrix}N&M\\I_n&S\end{bmatrix}^{-1}\in\Hfty^{n\times{}n}.
\end{equation}
\end{lemma}
\begin{proof}
First note that the transfer function from $d$ to $y$ is stable if and only if
\begin{equation}\label{eq:Gcond}
\funof{I-GS}^{-1}G\in\Hfty^{n\times{}n}.
\end{equation}
Furthermore
\[
\begin{aligned}
\begin{bmatrix}N&M\\I_n&S\end{bmatrix}&=\begin{bmatrix}N&M\\I_n&S\end{bmatrix}\begin{bmatrix}-S&0\\I_n&I_n\end{bmatrix}\begin{bmatrix}-S^{\mathsf{T}}&0\\S^{\mathsf{T}}&I_n\end{bmatrix},\\
&=\begin{bmatrix}M-NS&M\\0&S\end{bmatrix}\begin{bmatrix}-S^{\mathsf{T}}&0\\S^{\mathsf{T}}&I_n\end{bmatrix},\\
\end{aligned}
\]
which implies that
\begin{equation}\label{eq:eqivcond}
\text{\eqref{eq:stabcond}}\Longleftrightarrow{}\funof{M-NS}^{-1}\in\Hfty^{n\times{}n}.
\end{equation}
Since $N\in\Hfty^{n\times{}n}$ and $\funof{M-NS}^{-1}N=\funof{I_n-GS}^{-1}G$, \eqref{eq:eqivcond}$\implies$\eqref{eq:Gcond}. To show the converse implication, first note that 
\[
\funof{I_n-GS}^{-1}GS=I_n-\funof{I_n-GS}^{-1}.
\]
Therefore \eqref{eq:Gcond} implies that
\[
\funof{I_n-GS}^{-1}\funof{GX+Y}=\funof{M-NS}^{-1}\funof{NX+MY}
\]
is in $\Hfty^{n\times{}n}$ for any $X,Y\in\Hfty^{n\times{}n}$. Since $N$ and $M$ are left coprime factors, $X$ and $Y$ can be chosen such that $NX+MY=I_n$, and so \eqref{eq:Gcond}$\implies$\eqref{eq:eqivcond} as required.
\end{proof}

\Cref{lem:2} shows that in the input-output setting, stability of the cyclic feedback interconnection once again reduces to a question about invertibility of a matrix on the form of \eqref{eq:maininv} from \cref{sec:res1}. Under a few minor technical restrictions, we can turn this into a condition that must be satisfied pointwise in $s$. For example, given any positive $\gamma_1,\ldots{},\gamma_n$, if the inequality in condition (ii) of \Cref{thm:1} holds, then the dynamics in \eqref{eq:ipop} are asymptotically stable for all $G_k=M_k^{-1}N_k$, where $M_k,N_k\in\disc$ satisfy
\begin{equation}\label{eq:stab2}
N_k\s/M_k\s\in\delt{\gamma_k}{\mathrm{mob}}
\end{equation}
for all $s$ in the extended closed right half-plane (i.e. including $s=\cfunof{\infty}$). The converse of this condition can also be shown by making a construction analogous to that in the proof of \Cref{lem:1} based on rational functions. This can be done using standard interpolation results for rational functions (see for example \cite[Lemma 1.14]{Vin00}, which can also enforce additional requirements, such as making $G_k$ strictly proper).

\begin{remark}
That $G_k$ has a coprime factor representation that is continuous on on the extended imaginary axis is equivalent to requiring that $G_k$ be approximable to arbitrary precision (in the graph topology) with rational functions, see \cite[Theorem 7.1]{Vin00}.
\end{remark}

\subsection{Frequency domain and state-space tests}\label{sec:res3}

In the previous section we showed that the inequality in \Cref{thm:1}(ii) guarantees stability of the cyclic feedback interconnection for all subsystems with dynamics that satisfy a condition on the form
\begin{equation}\label{eq:condcond}
N_k\s/M_k\s\in\delt{\gamma_k}{\mathrm{mob}},
\end{equation}
for all $s$ in the (extended) closed right half-plane. In the autonomous case $N_k$ and $M_k$ were polynomials, and in the input-output case they were coprime functions in $\disc$. In this section we will turn the analysis around, and explain how to check whether a given pair of functions $N_k$ and $M_k$ satisfy \eqref{eq:condcond}. 

First note that for the classes of functions considered, $N_k\s/M_k\s$ is a well defined point in \Cext{} if and only if
\begin{equation}\label{eq:stabilizable}
N_k\s\neq{}0\;\;\text{or}\;\;M_k\s\neq{}0.
\end{equation}
This does not impose a restriction in the input-output case as described, since there $N_k,M_k\in\disc$ were assumed to be coprime, which guarantees that the above holds for all $s$ in the extended closed right half-plane. However no such assumption was imposed in the autonomous case. In this case we must assume that \eqref{eq:stabilizable} holds for all $s$ in the closed right half-plane, which is equivalent to requiring that the subsystem dynamics in \eqref{eq:subsys} are behaviorally stabilizable \cite[Theorem 5.2.30]{PW98}.

Assuming \eqref{eq:stabilizable}, next note that for $ad-bc\neq{}0$ and $\gamma>0$,
\[
w=\frac{\gamma\funof{az+b}}{cz+d}\;\;\Longleftrightarrow{}\;\;z=\frac{dw/\gamma-b}{-cw/\gamma+a}.
\]
If we set
\[
H_k=\funof{-cG_k/\gamma_k+a}^{-1}\funof{dG_k/\gamma_k-b},
\]
where $G_k=M_k^{-1}N_k$, it then follows that \eqref{eq:condcond} holds for all $s$ in the (extended) closed right half-plane if and only if
\begin{equation}\label{eq:condcondcond}
\norm{H_k}_\infty{}\leq{}1.
\end{equation}
For the classes of functions considered, \eqref{eq:condcondcond} can be checked with frequency domain techniques. More specifically \eqref{eq:condcondcond} holds if and only if $H_k$ is stable, and the Nyquist diagram of $H_k$ lies in the closed unit disk. This condition can be translated into a requirement in terms of $G_k$ as follows.
\begin{enumerate}
\item Plot the Nyquist diagram of $G_k$.
\item Map the unit disk through the M\"{o}bius transform $\tfrac{az+b}{cz+d}$. This will result in a generalised disk (the interior or exterior of a circle, or a half-plane).
\item Scale the resulting generalised disk so that it contains the Nyquist diagram. 
\item Check that $\frac{dG_k/\gamma-b}{a-cG_k/\gamma}$ is stable, where $\gamma>0$ is the scaling factor from (3). If it is, \eqref{eq:condcond} holds with $\gamma_k=\gamma$.
\end{enumerate}
If in addition $H_k$ is a real rational function, \eqref{eq:condcondcond} can be checked with state-space techniques using the Kalman-Yakubovic-Popov lemma. More specifically, if
\[
C\funof{sI_m-A}^{-1}B+D,
\]
is a minimal realisation of $H_k$, then, by \cite[Theorems 3 and 4]{Wil71}, \eqref{eq:condcondcond} holds if and only if there exists a positive semi-definite $P$ such that the matrix
\begin{equation}\label{eq:lmi}
\begin{bmatrix}
A^{\mathsf{T}}P+PA+C^{\mathsf{T}}C&C^{\mathsf{T}}D+PB\\D^{\mathsf{T}}C+B^{\mathsf{T}}P&D^{\mathsf{T}}D-1
\end{bmatrix}
\end{equation}
is negative semi-definite. Given any $A$, $B$, $C$ and $D$, the above can be checked using convex programming. A suitable $\gamma_k$ can then be found using a line search.

\begin{remark}\label{rem:nonlinear}
Although we have only considered the linear time-invariant case in this paper, \eqref{eq:condcondcond} indicates how the derived conditions might be extended into the nonlinear setting. In the input-output case, \eqref{eq:condcondcond} suggests that if we replace the transfer functions with operators, stability will be achieved if the operator $H_k$ is non-expansive. Furthermore for the cyclic interconnection of subsystems described by nonlinear ordinary differential equations, the dissipation inequality associated with \eqref{eq:lmi} might be used to derive criteria within the dissipativity framework of \cite{Wil72}.
\end{remark}

\subsection{Numerical example}\label{sec:res4}

\begin{figure}
\centering
\input{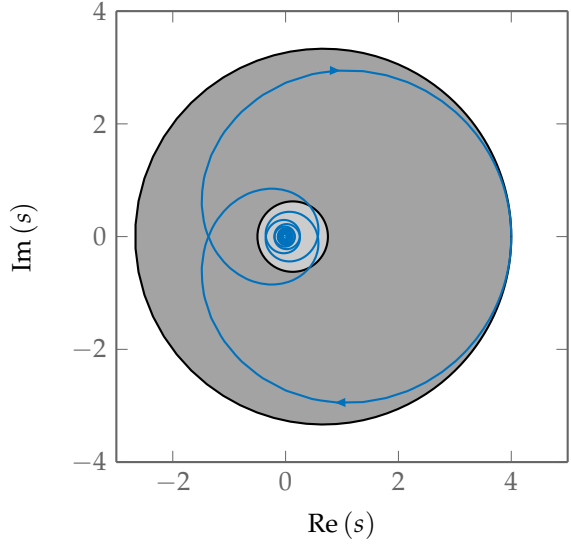}
\caption{\label{fig:2}Constructions used to determine $\gamma_1$ in \cref{sec:res4}.} 
\end{figure}

The following example illustrates how to apply the frequency domain test from the previous subsection. Suppose that $n=2$, $a=2$, $b=1$, $c=1$, $d=3$,
\[
G_1\s=\frac{4e^{-0.7s}}{s+1},\;\;\text{and}\;\;G_2\s=K,
\]
where $K>0$. In this case, the inequality from\Cref{thm:1} becomes
\[
-3\bar{\gamma}^2+8>0,
\]
so we are interested in finding the smallest values of $\gamma_1$ and $\gamma_2$ such that \eqref{eq:stab2} holds. \Cref{fig:2} illustrates how to determine $\gamma_1$. The blue curve is the Nyquist diagram of $G_1$, and the light grey disk the result of mapping the unit disk through the M\"{o}bius transform. Scaling up the light grey disk by a factor of 16/3 gives the dark grey disk. This is the smallest scaling factor required to completely contain the Nyquist diagram. Since the Nyquist diagram makes no encirclements of the point $16a/3c$, the stability check is passed, so $\gamma_1=16/3$. It can be similarly shown that $\gamma_2=4K/3$, meaning that
\[
\bar{\gamma}=8/3\sqrt{K}.
\]
It then follows from \eqref{eq:stab2} the cyclic feedback interconnection is input-output stable if $K<3/8$. This not only illustrates how to apply the frequency domain test, but also the advantage of the more flexible criterion. For example, the small gain theorem would require that $K<1/4$, and the secant criterion cannot be applied at all since $G_1$ is not output-strictly-passive.

\subsection{Connection to existing approaches}\label{sec:res5}

As noted in the introduction, several existing stability criteria for the cyclic feedback interconnection can be obtained from \Cref{thm:1}. When combined with the frequency domain or state-space techniques detailed above, the following choices of the M\"{o}bius transform
\[
f\funof{z}=\frac{az+b}{cz+d}
\]
correspond the linear time invariant cases of some well known theorems:
\begin{enumerate}
\item The small gain theorem corresponds to the case $f\funof{z}=z$. If this is the case, the inequality in (ii) simplifies to:
\[
\bar{\gamma}<1.
\]
\item The large gain theorem corresponds to the case $f\funof{z}=\tfrac{1}{z}$. If this is the case, the inequality in (ii) simplifies to:
\[
\bar{\gamma}>1.
\]
\item The secant criterion corresponds to the case $f\funof{z}=\tfrac{1}{2}\funof{z+1}$. If this is the case, the inequality in (ii) simplifies to:
\[
\bar{\gamma}<\sec\funof{\tfrac{\pi}{n}}.
\]
\item The secant like criterion based on input feedforward passivity \cite{YP10} corresponds to the case $f\funof{z}=\tfrac{2}{z+1}$. If this is the case, the inequality in (ii) simplifies to:
\[
\bar{\gamma}>\cos\funof{\tfrac{\pi}{n}}.
\]
\end{enumerate}

\begin{remark}
In all the above examples (and in the example in \cref{sec:res4}), it is clear whether large $\bar{\gamma}$ or small $\bar{\gamma}$ is desired. In such cases it makes sense to associate the subsystem dynamics with a notion of size based on the indices $\gamma_k$. For example, in the small gain theorem small $\gamma_k$ is desired. Therefore when finding the scaling of the generalised disk in (3) of the frequency domain test, or when searching for a $\gamma_k$ such that the convex program in \eqref{eq:lmi} is feasible, it makes sense to search for the smallest such $\gamma_k$. This then gives a notion of size tailored to the features of the specific M\"{o}bius transform (and coincides with the H-infinity norm of the transfer function in this case).

Although the inequality in \Cref{thm:1} yields many criteria with this property, this is not the case for every choice of $a$, $b$, $c$ and $d$. For example, if
\[
a=1,b=16,c=1,\;\text{and}\;d=9/8,
\]
when $n=3$ the inequality is satisfied if either $0<\bar{\gamma}<1/40$, or $\bar{\gamma}>1/24$. In such cases (which can only arise if $b^2-a^2>0$ and $d^2-c^2>0$) both the largest and smallest values of $\gamma_k$ are of interest, leading to a kind of mixed large gain/small gain analysis.
\end{remark}

\section{Conclusions}

In this paper stability criteria based on a secant like criterion have been presented. The results unify and generalise many existing conditions that can be used to assess the stability of the cyclic feedback interconnection of linear time-invariant single-input-single-output systems, including the Secant Criterion and the small gain theorem. A small example illustrating the utility of the more general criteria was given, and directions for nonlinear extensions briefly discussed.


\bibliographystyle{IEEEtran}
\bibliography{references.bib}

\end{document}